\documentclass[
aps,%
11pt,%
final,%
notitlepage,%
oneside,%
onecolumn,%
nobibnotes,%
nofootinbib,%
superscriptaddress,%
noshowpacs,%
centertags,%
showkeys]%
{revtex4}
\begin{document}
\title{Generalized Normal Forms of Two-Dimensional Autonomous Systems with
a Hamiltonian Unperturbed Part}
\author{V.~V.~Basov}
\email{vlvlbasov@rambler.ru}
\affiliation{Saint Petersburg State University}
\author{A.~S.~Vaganyan}
\email{armay@yandex.ru}
\thanks{This research is supported by the Chebyshev Laboratory (Department of 
Mathematics and Mechanics, St. Petersburg State University) under RF 
Government grant 11.G34.31.0026.}
\altaffiliation{Chebyshev Laboratory, St. Petersburg State University, 14th 
Line, 29b, Saint Petersburg, 199178 Russia} 
\affiliation{Saint Petersburg State University}
\begin{abstract}
Generalized pseudo-Hamiltonian normal forms (GPHNF) and an effective method of
obtaining them are introduced for two-dimensional systems of autonomous ODEs 
with a Hamiltonian quasi-homogeneous unperturbed part of an arbitrary degree. 
The terms that can be additionally eliminated in a GPHNF are constructively
distinguished, and it is shown that after removing them GPHNF becomes a 
generalized normal form (GNF). By using the introduced method, all the GNFs 
are obtained in cases where the unperturbed part of the system is Hamiltonian 
and has monomial components, which allowed to generalize some results by 
Takens, Baider and Sanders, as well as Basov et al.
\end{abstract}
\maketitle
\section{Some facts from the theory of normal forms}
\subsection{Quasi-homogeneous polynomials and vector fields}
Let $x=(x_1,x_2),$ $p=(p_1,p_2)$ and $D=(\partial_1,\partial_2)$ where 
$x_1,x_2$ are scalar variables, $p_1,p_2$ are nonnegative integers, 
and $\partial_1=\partial/\partial x_1,$ $\partial_2=\partial/\partial x_2.$

Denote by $\langle\,.\,,\,.\,\rangle$ the standard inner product on vectors, 
Euclidean or Hermitian depending on the context.

Consider vector $\gamma=(\gamma_1,\gamma_2)$ with $\gamma_1,\gamma_2\in
\mathbb N,$ $\text{GCD}(\gamma_1,\gamma_2)=1$ called \textit{weight} 
of the variable $x,$ and let $\delta=\gamma_1+\gamma_2.$
\begin{definition}[g.~d.]
\label{weight}\ \ We call the number $\langle p,\gamma\rangle$ 
\textit{the generalized degree} of a monomial $x_1^{p_1} x_2^{p_2}.$ 
\end{definition}
\begin{definition}[QHP]
\label{QHP}\ \
Polynomial $P(x)$ is called \textit{a quasi-homogeneous polynomial} of g.~d. 
$k$ with weight $\gamma$ and is denoted by $P^{[k]}_\gamma(x)$ if it is 
a linear combination of monomials of g.~d. $k.$
\end{definition}
\begin{definition}[VQHP]\ \
Vector polynomial ${\mathcal P}(x)=\bigl(P_{1}(x),P_{2}(x)\bigr)$
is called \textit{a vector quasi-homogeneous polynomial} of g.~d. $k$ with 
weight $\gamma$ and is denoted by ${\mathcal P}_\gamma^{[k]}(x),$ if its 
components $P_{1}(x),P_{2}(x)$ are QHPs of g.~d. $k+\gamma_1$ and 
$k+\gamma_2$ respectively.
\end{definition}
\begin{remark}\ \
Further, we identify each VQHP ${\mathcal P}_\gamma^{[k]}(x)=
\bigl(P_{\gamma,1}^{[k+\gamma_1]}(x),P_{\gamma,2}^{[k+\gamma_2]}(x)\bigr)$ 
with the corresponding vector field, i.~e., define the action of VQHPs on 
formal power series:
$${\mathcal P}_\gamma^{[k]}(f)(x)=P_{\gamma,1}^{[k+\gamma_1]}(x)\,\partial_1 
f(x)+P_{\gamma,2}^{[k+\gamma_2]}(x)\,\partial_2 f(x),$$
and define the Lie bracket of two VQHPs ${\mathcal P}_\gamma^{[k]}(x)= 
\bigl(P_{\gamma,1}^{[k+\gamma_1]}(x), P_{\gamma,2}^{[k+\gamma_2]}(x)\bigr)$ 
and ${\mathcal Q}_\gamma^{[l]}(x)=\bigl(Q_{\gamma,1}^{[l+\gamma_1]}(x), 
Q_{\gamma,2}^{[l+\gamma_2]}(x)\bigr):$
$$[{\mathcal P}_\gamma^{[k]}, {\mathcal Q}_\gamma^{[l]}](x)=
\bigl({\mathcal P}_{\gamma}^{[k]}(Q_{\gamma,1}^{[l+\gamma_1]})(x)-
{\mathcal Q}_\gamma^{[l]}(P_{\gamma,1}^{[k+\gamma_1]})(x),\
{\mathcal P}_{\gamma}^{[k]}(Q_{\gamma,2}^{[l+\gamma_2]})(x)-
{\mathcal Q}_\gamma^{[l]}(P_{\gamma,2}^{[k+\gamma_2]})(x)\bigr).$$
\end{remark}
\begin{definition}\ \ 
Given a vector series $\mathcal P(x)=\sum_{k=0}^\infty 
\mathcal P_\gamma^{[k]}(x),$ we call the least positive 
integer $\chi\ge0$ such that $\mathcal P_\gamma^{[\chi]}\not\equiv0$
\textit{the generalized order} of $\mathcal P.$ Denote the generalized order 
by $\mathrm{ord}_\gamma \mathcal P.$ If $\mathcal P\equiv0,$ then 
$\mathrm{ord}_\gamma \mathcal P=+\infty.$
\end{definition}
\subsection{Homological equation}
Consider the system
\begin{equation}
\label{Eq1}
\dot x={\mathcal P}_{\gamma}^{[\chi]}(x)+\mathcal X(x)\qquad\bigl(\chi\ge0,\ 
\mathcal X=(X_1,X_2),\ \mathrm{ord}_\gamma \mathcal X\ge\chi+1\bigr).
\end{equation}
Let the near-identity formal change of variables
\begin{equation}
\label{Eq2}
x=y+\mathcal Q(y)
\qquad\bigl(y=(y_1,y_2),\ \mathcal Q=(Q_1,Q_2),\ 
\mathrm{ord}_\gamma\mathcal Q\ge1\bigr)
\end{equation}
transform it into the system
\begin{equation}
\label{Eq3}
\dot y={\mathcal P}_{\gamma}^{[\chi]}(y)+\mathcal Y(y)\qquad\bigl(
\mathcal Y=(Y_1,Y_2),\ \mathrm{ord}_\gamma \mathcal Y\ge\chi+1\bigr).
\end{equation}
Then the series $\mathcal X,\mathcal Y$ and $\mathcal Q$ satisfy \textit{the homological equation} (see~\cite{basov03}):
\begin{equation}
\label{Eq4}
[\mathcal P_{\gamma}^{[\chi]},\mathcal Q_{\gamma}^{[k]}]=\widetilde{%
\mathcal Y}_{\gamma}^{[k+\chi]}-\mathcal Y_{\gamma}^{[k+\chi]}
\end{equation}
where $k\ge1,$ and $\widetilde{\mathcal Y}_{\gamma}^{[k+\chi]}$ includes only 
the components of the VQHPs $\mathcal Q_{\gamma}^{[l]}$ and 
$\mathcal Y_{\gamma}^{[l+\chi]}$ with $l=\overline{1,k-1}.$
If $\mathrm{ord}_\gamma \mathcal Q=m\ge1,$ then for $k=\overline{1,m},$ 
equation~\eqref{Eq4} takes the form
\begin{equation}
\label{Eq5}
\mathcal X_{\gamma}^{[l+\chi]}=\mathcal Y_{\gamma}^{[l+\chi]}\quad 
(l=\overline{1,m-1}),\qquad[\mathcal P_{\gamma}^{[\chi]},
\mathcal Q_{\gamma}^{[m]}]=\mathcal X_{\gamma}^{[m+\chi]}-
\mathcal Y_{\gamma}^{[m+\chi]}.
\end{equation}
\subsection{Generalized normal form (GNF)}
Denote the linear spaces of VQHPs of g.~d. $k\ge1$ and $k+\chi$ in $x$ by 
$\mathfrak V_\gamma^{[k]}$ and $\mathfrak V_\gamma^{[k+\chi]},$ and their 
dimensions by $k_\gamma=\dim\mathfrak V_\gamma^{[k]}$ and $(k+\chi)_\gamma= 
\dim\mathfrak V_\gamma^{[k+\chi]}.$

In the spaces $\mathfrak V_\gamma^{[k]}$ and $\mathfrak V_\gamma^{[k+\chi]},$
choose the standard bases $\{(x_1^{p_1}x_2^{p_2},0),(0,x_1^{p'_1}x_2^{p'_2}):\ 
\langle p,\gamma\rangle-\gamma_1=\langle p',\gamma\rangle-\gamma_2=q\},$ 
where $q$ equals $k$ and $k+\chi$ respectively, order their elements
lexicographically, and denote the corresponding matrix representation of the 
linear operator $[\mathcal P_{\gamma}^{[\chi]},\,.\,]: 
\mathfrak V_\gamma^{[k]}\rightarrow\mathfrak V_\gamma^{[k+\chi]}$ by 
$\mathbf P_{\gamma,k}^{[\chi]}.$

Assume that the matrix $\mathbf P_{\gamma,k}^{[\chi]}$ has rank 
$r_k=k_\gamma-k_\gamma^0$ where $0\le k_\gamma^0\le k_\gamma-1.$

Comparing the VQHPs $\mathcal Q_{\gamma}^{[k]},\ \widetilde{%
\mathcal Y}_{\gamma}^{[k+\chi]}$ and $\mathcal Y_{\gamma}^{[k+\chi]}$ 
to the coefficient vectors of their expansions in the standard bases 
$\mathbf Q_{\gamma}^{[k]},\ \widetilde{\mathbf Y}_{\gamma}^{[k+\chi]}$ and 
$\mathbf Y_{\gamma}^{[k+\chi]},$ represent the homological 
equation~\eqref{Eq4} as a system of linear algebraic equations:
$\mathbf P_{\gamma,k}^{[\chi]}\mathbf Q_{\gamma}^{[k]}=
\widetilde{\mathbf Y}_{\gamma}^{[k+\chi]}-\mathbf Y_{\gamma}^{[k+\chi]}.$

From the obtained equations, select a subsystem of order $r_k$ with nonzero 
determinant, and solve it with respect to $\mathbf Q_{\gamma}^{[k]}.$

Arbitrarily fix the remaining $k_\gamma^0$ coefficients of 
$\mathbf Q_{\gamma}^{[k]}.$ 

Substituting $\mathbf Q_{\gamma}^{[k]}$ in the remaining equations, we obtain 
$n_k=(k+\chi)_\gamma-r_k$ linearly independent relations on coefficients of
the VQHP $\mathcal Y_{\gamma}^{[k+\chi]}$ called \textit{resonance equations}:
$$\langle\mathbf a_i^{k},\mathbf Y_{\gamma}^{[k+\chi]}\rangle=
\langle\mathbf a_i^{k},\widetilde{\mathbf Y}_{\gamma}^{[k+\chi]}\rangle,\quad
\mathbf a_i^{k}=\mathrm{const}\qquad(i=\overline{1,n_k}).$$

Coefficients of the VQHP $\mathcal Y_{\gamma}^{[k+\chi]}$ that really 
participate in at least one of the resonance equations are called 
\textit{resonant} and the others are called \textit{nonresonant}.
\begin{definition}\ \ 
We call any set $\mathfrak Y=\cup_{k=1}^\infty\{
\mathbf Y_{\gamma,l_j}^{[k+\chi]}\}_{j=1}^{n_k}$ of resonant coefficients 
of the VQHP $\mathcal Y_{\gamma}^{[k+\chi]}$ such that for all $k\ge1$ 
$\det(\{\mathbf a_{i,l_j}^{k}\}_{i,j=1}^{n_k})\ne0$ \textit{a resonant set}.
\end{definition}
\begin{definition}[GNF]\ \ 
System~\eqref{Eq3} is called \textit{a generalized normal form} if all
coefficients of the vector series $\mathcal Y$ are zero except for the
coefficients from some resonant set $\mathfrak Y$ that have arbitrary 
values.
\end{definition}

Thereby, each GNF is generated by a certain resonant set $\mathfrak Y.$
\begin{proposition}[{\cite[Th.~2]{basov03}}]\ \ 
Given any system~\eqref{Eq1} and arbitrarily chosen resonant set 
$\mathfrak Y,$ then there exists a near-identity formal 
transformation~\eqref{Eq2} that brings it to a certain GNF~\eqref{Eq3} 
generated by $\mathfrak Y.$
\end{proposition}
\section{Generalized pseudo-Hamiltonian normal form (GPHNF)}
\subsection{Euler operator}

\begin{definition}\ \
\textit{The Euler operator} with weight $\gamma$ is the VQHP 
$\mathcal E_\gamma(x)=(\gamma_1 x_1,\gamma_2 x_2).$
\end{definition}
Clearly, the Euler operator has zero g.~d., so in the designation 
$\mathcal E_\gamma,$ the g.~d. is omitted.

Let $Q_\gamma^{[k]}$ be a QHP, and ${\mathcal Q}_\gamma^{[k]}= 
\bigl(Q_{\gamma,1}^{[k+\gamma_1]},Q_{\gamma,2}^{[k+\gamma_2]}\bigr)$ be a VQHP 
of g.~d. $k\ge0.$
\begin{proposition}\ \ 
The Euler operator with weight $\gamma$ has the following properties:
$$1)\ \mathcal E_\gamma(Q_\gamma^{[k]})=k Q_\gamma^{[k]};\quad
2)\ [\mathcal E_\gamma,\mathcal Q_\gamma^{[k]}]=k\mathcal Q_\gamma^{[k]};\quad
3)\ \mathrm{div}(Q_\gamma^{[k]}\mathcal E_\gamma)=(k+\delta)Q_\gamma^{[k]}.$$
\end{proposition}
\begin{proof}\ \
1)\ The property is obvious.

2)\ According to 1), $[\mathcal E_\gamma,\mathcal Q_\gamma^{[k]}]=
\bigl({\mathcal E}_{\gamma}(Q_{\gamma,1}^{[k+\gamma_1]})-
{\mathcal Q}_\gamma^{[k]}(\gamma_1 x_1),\
{\mathcal E}_{\gamma}(Q_{\gamma,2}^{[k+\gamma_2]})-
{\mathcal Q}_\gamma^{[k]}(\gamma_2 x_2)\bigr)=k\mathcal Q_\gamma^{[k]}.$

3)\ According to 1) and the Leibniz rule,
$\mathrm{div}(Q_\gamma^{[k]}\mathcal E_\gamma)=\mathcal E_\gamma(Q_\gamma^{[k]})+
Q_\gamma^{[k]} \mathrm{div}\mathcal E_\gamma=(k+\delta)Q_\gamma^{[k]}.$
\end{proof}
\begin{lemma}\label{EulerLm}\ \ 
Each VQHP $\mathcal Q_\gamma^{[k]}= \bigl(Q_{\gamma,1}^{[k+\gamma_1]}, 
Q_{\gamma,2}^{[k+\gamma_2]}\bigr)$ can be uniquely represented in the form
\begin{equation}
\label{Eq6}
\mathcal Q_\gamma^{[k]}=\mathcal I_\gamma^{[k]}+
J_\gamma^{[k]}\mathcal E_\gamma\qquad \bigl(
\mathcal I_\gamma^{[k]}= (-\partial_2 I_\gamma^{[k+\delta]},
\partial_1 I_\gamma^{[k+\delta]})\bigr).
\end{equation}
Herewith, $J_\gamma^{[k]}=\mathrm{div}(\mathcal Q_\gamma^{[k]})/(k+\delta).$
\end{lemma}
\begin{proof}\ \
Take the divergence of both sides of equality~\eqref{Eq6}. According to
property~3) of $\mathcal E_\gamma$ and the fact that 
$\mathrm{div}(\mathcal I_\gamma^{[k]})\equiv0,$ we obtain
$J_\gamma^{[k]}=\mathrm{div}(\mathcal Q_\gamma^{[k]})/(k+\delta).$ Hence, 
$I_\gamma^{[k+\delta]}$ must satisfy the system of equations 
$\partial_1 I_\gamma^{[k+\delta]}=Q_{\gamma,2}^{[k+\gamma_2]}-\gamma_2 x_2 
\mathrm{div}(\mathcal Q_\gamma^{[k]})/(k+\delta),$ 
$\partial_2 I_\gamma^{[k+\delta]}=\gamma_1 x_1 \mathrm{div}(
\mathcal Q_\gamma^{[k]})/(k+\delta)-Q_{\gamma,1}^{[k+\gamma_1]}.$

By the Poincare lemma, taking into account the quasi-homogeneity,
$I_\gamma^{[k+\delta]}$ exists and is unique.
\end{proof}
\subsection{Hamiltonian resonant sets}
Let $H_\gamma^{[\chi+\delta]}\not\equiv0$ be a QHP, and
$\mathcal H_\gamma^{[\chi]}=
(-\partial_2 H_\gamma^{[\chi+\delta]},\partial_1 H_\gamma^{[\chi+\delta]})$ 
be a VQHP of g.~d. $\chi\ge0.$

Denote the space of polynomials in $x$ by $\mathfrak{P},$ and
introduce the inner-product on $\mathfrak{P}:$
$$\langle\langle P,Q\rangle\rangle=P(D)\overline{Q}(x)|_{x=0}\qquad
\bigl(P,Q\in\mathfrak{P},\ D=(\partial_1,\partial_2)\bigr)$$
where the upper bar stands for the complex conjugation of the coefficients.
Then the operator $\left.\mathcal H_\gamma^{[\chi]}\right.^*: 
\mathfrak{P}\rightarrow\mathfrak{P}$ conjugate to 
$\mathcal H_\gamma^{[\chi]}$ with respect to 
$\langle\langle\,.\,,\,.\,\rangle\rangle$ has the form 
(see~\cite{BV})
\begin{equation}
\label{Eq7}
\left.\mathcal H_\gamma^{[\chi]}\right.^*=
x_2\cdot{\partial_1 \overline{H}_\gamma^{[\chi+\delta]}}(D)-
x_1\cdot{\partial_2 \overline{H}_\gamma^{[\chi+\delta]}}(D)=
{\partial_1 \overline{H}_\gamma^{[\chi+\delta]}}(D)\cdot x_2-
{\partial_2 \overline{H}_\gamma^{[\chi+\delta]}}(D)\cdot x_1.
\end{equation}
\begin{definition}
\label{REdef}\ \
Kernel $\mathfrak{R}_\gamma=
\mathrm{Ker}\left.\mathcal H_\gamma^{[\chi]}\right.^*$ 
of the operator~\eqref{Eq7} is called the space of \textit{resonant 
polynomials}. Denote the linear space of resonant QHPs of g.~d. $k\ge0$ by 
$\mathfrak{R}_\gamma^{[k]}.$
\end{definition}

Let $\{R_{\gamma,i}^{[k]}\}_{i=1}^{s_k}$ be a basis for 
$\mathfrak{R}_\gamma^{[k]},$ and 
$\{\widetilde R_{\gamma,i}^{[k]}\}_{i=1}^{\widetilde s_k}$ be a basis for
$\widetilde{\mathfrak{R}}_\gamma^{[k]}=\mathfrak{R}_\gamma^{[k]}\cap
\left.\mathcal H_\gamma^{[\chi]}\right.^*\!(\mathfrak P_\gamma^{[k+\chi]}).$
\begin{definition}\label{Nabor}\ \
Sets of QHPs $\mathfrak{S}_\gamma^{[k]}=\{S_{\gamma,j}^{[k]}\}_{j=1}^{s_k}$ and
$\widetilde{\mathfrak{S}}_\gamma^{[k]}=
\{\widetilde S_{\gamma,j}^{[k]}\}_{j=1}^{\widetilde s_k}$
are called, respectively, \textit{a Hamiltonian resonant set} and 
\textit{a Hamiltonian reduced resonant set} in g.~d. $k,$
if $\det(\{\langle\langle 
R_{\gamma,i}^{[k]},S_{\gamma,j}^{[k]}\rangle\rangle\}_{i,j=1}^{s_k})\ne0$ and 
$\det(\{\langle\langle\widetilde R_{\gamma,i}^{[k]},
\widetilde S_{\gamma,j}^{[k]}\rangle\rangle\}_{i,j=1}^{\widetilde s_k})\ne0.$
The sets $\mathfrak{S}_\gamma=\cup_{k\ge0}\mathfrak{S}_\gamma^{[k]}$ and 
$\widetilde{\mathfrak{S}}_\gamma=\cup_{k\ge0}
\widetilde{\mathfrak{S}}_\gamma^{[k]}$ are called, respectively, a Hamiltonian 
resonant set and a Hamiltonian reduced resonant set.
\end{definition}

Show the independence of the definition of a Hamiltonian resonant set of the 
choice of basis.

Since any other basis $\{{R'}_{\gamma,i}^{[k]}\}_{i=1}^{s_k}$ 
for the space $\mathfrak{R}_\gamma^{[k]}$ is obtained from 
$\{R_{\gamma,i}^{[k]}\}_{i=1}^{s_k}$ by multiplication by a nonsingular matrix:
${R'}_{\gamma,i}^{[k]}=\sum_{j=1}^{s_k}A_{ij} R_{\gamma,j}^{[k]},$ 
$\det{A}\ne0$ $(i=\overline{1,s_k}),$ for each Hamiltonian resonant set 
$\mathfrak{S}_\gamma^{[k]}=\{S_{\gamma,j}^{[k]}\}_{j=1}^{s_k}$ in g.~d. $k,$ 
we have $\det(\{\langle\langle 
{R'}_{\gamma,i}^{[k]},S_{\gamma,j}^{[k]}\rangle\rangle\}_{i,j=1}^{s_k})= 
\det(\{\sum_{l=1}^{s_k}A_{il}\langle\langle R_{\gamma,l}^{[k]}, 
S_{\gamma,j}^{[k]}\rangle\rangle\}_{i,j=1}^{s_k})= \det{A}\,
\det(\{\langle\langle R_{\gamma,i}^{[k]}, 
S_{\gamma,j}^{[k]}\rangle\rangle\}_{i,j=1}^{s_k})\ne0.$

Similarly, the definition of a Hamiltonian reduced resonant set 
$\widetilde{\mathfrak{S}}_\gamma^{[k]}$ is also independent of the choice of 
basis.
\begin{definition}\label{MinNabor}\ \
We say that a Hamiltonian resonant set ${\mathfrak{S}}_\gamma$ or a Hamiltonian reduced resonant set $\widetilde{\mathfrak{S}}_\gamma$ is  
\textit{minimal} if it consists of monomials.
\end{definition}
\subsection{Definition of the GPHNF and existence theorem for the normalizing transformation}
Let the unperturbed part of system~\eqref{Eq1} have the form 
$\mathcal P_{\gamma}^{[\chi]}= \mathcal H_\gamma^{[\chi]}$ where, as in the 
previous subsection, $\mathcal H_\gamma^{[\chi]}=
(-\partial_2 H_\gamma^{[\chi+\delta]}, \partial_1 H_\gamma^{[\chi+\delta]})$ 
is a VQHP of g.~d. $\chi\ge0.$

According to Lemma~\ref{EulerLm} system~\eqref{Eq1} can be uniquely 
represented in the form
\begin{equation}
\label{Eq8}
\dot x=\mathcal H_\gamma^{[\chi]}(x)+\sum_{k=1}^\infty \big(
\mathcal F_\gamma^{[k+\chi]}(x)+G_\gamma^{[k+\chi]}(x)
\mathcal E_\gamma(x)\big)
\end{equation}
where $\mathcal F_\gamma^{[k+\chi]}=
\big(-\partial_2 F_\gamma^{[k+\chi+\delta]},
\partial_1 F_\gamma^{[k+\chi+\delta]}\big).$ Herewith,
$G_\gamma^{[k+\chi]}=\mathrm{div}\big(
\mathcal X_\gamma^{[k+\chi]}\big)/(k+\chi+\delta).$
\begin{definition}[GPHNF]\ \
System~\eqref{Eq8} is called \textit{a generalized pseudo-Hamiltonian normal 
form}, if for some Hamiltonian resonant set $\mathfrak S_\gamma$ and 
Hamiltonian reduced resonant set $\widetilde{\mathfrak S}_\gamma,$ the 
following conditions are satisfied:
$$\forall\,k\ge1\qquad F_\gamma^{[k+\chi+\delta]}\in Lin(
\widetilde{\mathfrak S}_\gamma^{[k+\chi+\delta]}),\quad
G_\gamma^{[k+\chi]}\in Lin(\mathfrak S_\gamma^{[k+\chi]}).$$
\end{definition}

Thereby, the structure of each GPHNF is generated by certain Hamiltonian 
resonant and reduced resonant sets $\mathfrak S_\gamma$ and 
$\widetilde{\mathfrak S}_\gamma.$
\begin{lemma}\ \ 
1) Let $J_\gamma^{[k]}$ be a QHP of g.~d. $k\ge0.$ Then for some QHP 
$K_\gamma^{[k+\chi+\delta]},$
\begin{equation}
\label{Eq9}
[\mathcal H_\gamma^{[\chi]},J_\gamma^{[k]}\mathcal E_\gamma]=
\mathcal K_\gamma^{[k+\chi]}+\frac{k+\delta}{k+\chi+\delta}
\mathcal H_\gamma^{[\chi]}(J_\gamma^{[k]})\mathcal E_\gamma\qquad
\bigl(\mathcal K_\gamma^{[k+\chi]}=(-\partial_2 K_\gamma^{[k+\chi+\delta]},
\partial_1 K_\gamma^{[k+\chi+\delta]})\bigr).
\end{equation}
Herewith, if $\mathcal H_\gamma^{[\chi]}(J_\gamma^{[k]})=0,$ then
$\mathcal K_\gamma^{[k+\chi]}=-\chi J_\gamma^{[k]}\mathcal H_\gamma^{[\chi]}$
and $\mathcal H_\gamma^{[\chi]}(K_\gamma^{[k+\chi+\delta]})=0.$

2) Conversely, given any QHP $K_\gamma^{[k+\chi+\delta]}$ such that
$\mathcal H_\gamma^{[\chi]}(K_\gamma^{[k+\chi+\delta]})=0,$ then there exists 
unique QHP $J_\gamma^{[k]}$ such that 
$\mathcal H_\gamma^{[\chi]}(J_\gamma^{[k]})=0$ and equality~\eqref{Eq9} is 
satisfied.
\end{lemma}
\begin{proof}\ \ 
1) By Lemma~\ref{EulerLm} and the identity $\mathrm{div}(J_\gamma^{[k]}
\mathcal H_\gamma^{[\chi]})\equiv\mathcal H_\gamma^{[\chi]}(J_\gamma^{[k]}),$ 
it follows that
\begin{multline}\notag
[\mathcal H_\gamma^{[\chi]},J_\gamma^{[k]}\mathcal E_\gamma]=
\mathcal H_\gamma^{[\chi]}(J_\gamma^{[k]})\mathcal E_\gamma+
J_\gamma^{[k]}[\mathcal H_\gamma^{[\chi]},\mathcal E_\gamma]=
\mathcal H_\gamma^{[\chi]}(J_\gamma^{[k]})\mathcal E_\gamma-
\chi J_\gamma^{[k]}\mathcal H_\gamma^{[\chi]}=\\
=\Bigl(1-\frac{\chi}{k+\chi+\delta}\Bigr)
\mathcal H_\gamma^{[\chi]}(J_\gamma^{[k]})\mathcal E_\gamma+
\mathcal K_\gamma^{[k+\chi]}= \frac{k+\delta}{k+\chi+\delta}
\mathcal H_\gamma^{[\chi]}(J_\gamma^{[k]})\mathcal E_\gamma+
\mathcal K_\gamma^{[k+\chi]}.
\end{multline}
In particular, if $\mathcal H_\gamma^{[\chi]}(J_\gamma^{[k]})=0,$ then
$\mathcal K_\gamma^{[k+\chi]}=-\chi J_\gamma^{[k]}\mathcal H_\gamma^{[\chi]}.$
Herewith, from the chain of equalities 
$[\mathcal H_\gamma^{[\chi]},J_\gamma^{[k]}\mathcal H_\gamma^{[\chi]}]=
\mathcal H_\gamma^{[\chi]}(J_\gamma^{[k]})\mathcal H_\gamma^{[\chi]}+
J_\gamma^{[k]}[\mathcal H_\gamma^{[\chi]},\mathcal H_\gamma^{[\chi]}]=0,$
it follows that $\mathcal H_\gamma^{[\chi]}(K_\gamma^{[k+\chi+\delta]})=0.$

2) The converse follows from the fact that any two integrals of 
$\mathcal H_\gamma^{[\chi]}$ are functionally dependent.
\end{proof}
\begin{theorem}\ \ Given any system~\eqref{Eq8} and arbitrary Hamiltonian 
resonant set $\mathfrak S_\gamma$ and Hamiltonian reduced resonant set 
$\widetilde{\mathfrak S}_\gamma,$ then there exists a near-identity formal 
transformation~\eqref{Eq2} that brings it to a certain GPHNF generated by 
$\mathfrak S_\gamma$ and $\widetilde{\mathfrak S}_\gamma.$
\end{theorem}
\begin{proof}\ \ 
Let $\mathfrak S_\gamma,\ \widetilde{\mathfrak S}_\gamma$ be Hamiltonian 
resonant and reduced resonant sets for $\mathcal H_\gamma^{[\chi]},$ and let
transformation~\eqref{Eq2} with $\mathrm{ord}_\gamma \mathcal Q=m\ge1$ 
bring system~\eqref{Eq8} into system~\eqref{Eq3} of the form
\begin{equation}
\label{Eq10}
\dot y=\mathcal H_\gamma^{[\chi]}(y)+\sum_{k=1}^\infty
\big(\widetilde{\mathcal F}_\gamma^{[k+\chi]}(y)+
\widetilde G_\gamma^{[k+\chi]}(y)\mathcal E_\gamma(y)\big),
\end{equation}
where $\widetilde{\mathcal F}_\gamma^{[k+\chi]}=
(-\partial_2 \widetilde F_\gamma^{[k+\chi+\delta]},
\partial_1 \widetilde F_\gamma^{[k+\chi+\delta]})$ and 
$\widetilde G_\gamma^{[k+\chi]}=
\mathrm{div}(\mathcal Y_\gamma^{[k+\chi]})/(k+\chi+\delta).$

By formulas~\eqref{Eq6} and~\eqref{Eq9}, we get the following expression for 
the Lie bracket of the VQHPs $\mathcal H_{\gamma}^{[\chi]}$ and 
$\mathcal Q_{\gamma}^{[m]}:$
\begin{equation}
\label{Eq11}
[\mathcal H_{\gamma}^{[\chi]},\mathcal Q_{\gamma}^{[m]}]=
[\mathcal H_{\gamma}^{[\chi]},\mathcal I_\gamma^{[m]}+
J_\gamma^{[m]}\mathcal E_\gamma]=
\bigl([\mathcal H_{\gamma}^{[\chi]},\mathcal I_\gamma^{[m]}]+
\mathcal K_\gamma^{[m+\chi]}\bigr)+
\frac{m+\delta}{m+\chi+\delta}
\mathcal H_\gamma^{[\chi]}(J_\gamma^{[m]})\mathcal E_\gamma.
\end{equation}

Substituting into~\eqref{Eq5} $\mathcal P_{\gamma}^{[\chi]}=
\mathcal H_\gamma^{[\chi]},$ $\mathcal X_{\gamma}^{[m+\chi]}=
\mathcal F_\gamma^{[m+\chi]}+G_\gamma^{[m+\chi]}\mathcal E_\gamma,$ and
$\mathcal Y_{\gamma}^{[m+\chi]}=\widetilde{\mathcal F}_\gamma^{[m+\chi]}+
\widetilde G_\gamma^{[m+\chi]}\mathcal E_\gamma$ and using 
equality~\eqref{Eq11}, we get the system of equations
\begin{equation}
\label{Eq12}
\mathcal H_{\gamma}^{[\chi]}(I_\gamma^{[m+\delta]})+K_\gamma^{[m+\chi+\delta]}=
{F}_\gamma^{[m+\chi+\delta]}-\widetilde{F}_\gamma^{[m+\chi+\delta]},\quad
\frac{m+\delta}{m+\chi+\delta}\mathcal H_\gamma^{[\chi]}(J_\gamma^{[m]})=
G_\gamma^{[m+\chi]}-\widetilde G_\gamma^{[m+\chi]}.
\end{equation}

Denote ${\mathfrak{S}}_\gamma^{[m+\chi]}=
\{S_{\gamma,j}^{[m+\chi]}\}_{j=1}^{s_{m+\chi}},$
$\widetilde{\mathfrak{S}}_\gamma^{[m+\chi+\delta]}=
\{\widetilde S_{\gamma,j}^{[m+\chi+\delta]}\}_{j=1}^{%
\widetilde s_{m+\chi+\delta}}.$

In the linear spaces $\mathfrak{R}_\gamma^{[m+\chi]}$ and
$\widetilde{\mathfrak{R}}_\gamma^{[m+\chi+\delta]}=
\mathfrak{R}_\gamma^{[m+\chi+\delta]}\cap
\left.\mathcal H_\gamma^{[\chi]}\right.^*\!(
\mathfrak P_\gamma^{[m+2\chi+\delta]}),$ choose bases 
$\{R_{\gamma,i}^{[m+\chi]}\}_{i=1}^{ s_{m+\chi}}$ and 
$\{\widetilde R_{\gamma,i}^{[m+\chi+\delta]}\}_{i=1}^{%
\widetilde s_{m+\chi+\delta}}$ respectively.

Define matrices $A=\{\langle\langle R_{\gamma,i}^{[m+\chi]},
S_{\gamma,j}^{[m+\chi]}\rangle\rangle\}_{i,j=1}^{ s_{m+\chi}}$ and
$\widetilde A=\{\langle\langle \widetilde R_{\gamma,i}^{[m+\chi+\delta]},
\widetilde S_{\gamma,j}^{[m+\chi+\delta]}\rangle\rangle\}_{i,j=1}^{%
\widetilde  s_{m+\chi+\delta}},$ vectors 
$c=\{\langle\langle R_{\gamma,i}^{[m+\chi]},
G_\gamma^{[m+\chi]}\rangle\rangle\}_{i=1}^{ s_{m+\chi}}$ and
$\widetilde c=\{\langle\langle \widetilde R_{\gamma,i}^{[m+\chi+\delta]},
F_\gamma^{[m+\chi+\delta]}\rangle\rangle\}_{i=1}^{%
\widetilde s_{m+\chi+\delta}},$ $b=A^{-1}c$ and $\widetilde b=
\widetilde A^{-1}\widetilde c,$ and QHPs
$\widetilde G_\gamma^{[m+\chi]}=\sum_{j=1}^{ s_{m+\chi}} b_j 
S_{\gamma,j}^{[m+\chi]}$ and $\widetilde F_\gamma^{[m+\chi+\delta]}=
\sum_{j=1}^{ \widetilde s_{m+\chi+\delta}}\widetilde b_j 
\widetilde S_{\gamma,j}^{[m+\chi+\delta]}.$

Then $\langle\langle R_{\gamma,i}^{[m+\chi]},
\widetilde G_\gamma^{[m+\chi]}-G_\gamma^{[m+\chi]}\rangle\rangle=
\sum_{j=1}^{ s_{m+\chi}}A_{ij}b_j-c_i=0$ for all $i=\overline{1, s_{m+\chi}}.$

Hence, by the Fredholm alternative, we obtain the QHP $J_\gamma^{[m]}$ 
satisfying~$(\ref{Eq12}_2),$ up to an integral of the VQHP 
$\mathcal H_\gamma^{[\chi]}.$

To complete the proof, it is enough to consider the case where
$\widetilde G_\gamma^{[m+\chi]}=G_\gamma^{[m+\chi]}.$

We have $\langle\langle\widetilde R_{\gamma,i}^{[m+\chi+\delta]},
\widetilde F_\gamma^{[m+\chi+\delta]}-F_\gamma^{[m+\chi+\delta]}\rangle\rangle=
\sum_{j=1}^{\widetilde s_{m+\chi+\delta}}\widetilde A_{ij}
\widetilde b_j-\widetilde c_i=0$ for all 
$i=\overline{1,\widetilde s_{m+\chi+\delta}}.$ Hence, by the Fredholm 
alternative and the definition of $\widetilde R_{\gamma,i}^{[m+\chi+\delta]},$ 
the QHP $\widetilde F_\gamma^{[m+\chi+\delta]}-F_\gamma^{[m+\chi+\delta]}$ 
can be represenred in the form~$(\ref{Eq12}_1)$ where 
$\mathcal H_\gamma^{[\chi]}(K_\gamma^{[m+\chi+\delta]})=0.$ And according to 
Lemma~2, every such QHP $K_\gamma^{[m+\chi+\delta]}$ can be obtained by 
choosing an appropriate $J_\gamma^{[m]}$ in~\eqref{Eq11} such that 
$\mathcal H_\gamma^{[\chi]}(J_\gamma^{[m]})=0.$

So, we have proved the existense of the QHPs $I_\gamma^{[m+\delta]}$ and 
$J_\gamma^{[m]}$ that satisfy~\eqref{Eq12}, and hence, there exists a 
VQHP $\mathcal Q_\gamma^{[m]}$ such that the transformation~\eqref{Eq2} with 
$\mathrm{ord}_\gamma \mathcal Q=m\ge1$ takes system~\eqref{Eq8} to the 
view~\eqref{Eq10} with $\widetilde F_\gamma^{[m+\chi+\delta]}\in 
Lin(\widetilde{\mathfrak S}_\gamma^{[m+\chi+\delta]})$ and
$\widetilde G_\gamma^{[m+\chi]}\in Lin(\mathfrak S_\gamma^{[m+\chi]}).$

Hence, step by step increasing $m,$ we find the required transformation as a 
composition of the transformations obtained on each step.
\end{proof}
\section{Reduction of the GPHNF to GNF}
Consider GPHNF~\eqref{Eq10}
$\dot y=\mathcal H_\gamma^{[\chi]}(y)+\sum_{k=1}^\infty 
\big(\widetilde{\mathcal F}_\gamma^{[k+\chi]}(y)+
\widetilde G_\gamma^{[k+\chi]}(y)\mathcal E_\gamma(y)\big)$
generated by arbitrary minimal sets 
$\mathfrak S_\gamma,\widetilde{\mathfrak S}_\gamma.$

Coefficients of its perturbation $\mathcal Y$ can be expressed in terms of 
$\widetilde F$ and $\widetilde G$ as follows:
\begin{equation}
\label{Eq13}
Y_1^{(p_1+1,p_2)}=-(p_2+1)\widetilde{F}^{(p_1+1,p_2+1)}+
\gamma_1 \widetilde{G}^{(p_1,p_2)},\quad
Y_2^{(p_1,p_2+1)}=(p_1+1)\widetilde{F}^{(p_1+1,p_2+1)}+
\gamma_2 \widetilde{G}^{(p_1,p_2)}.
\end{equation}

Note that if $y_1^{p_1} y_2^{p_2}\not\in \mathfrak S_\gamma,$ then there 
exists a vector $q=(q_1,q_2)$ with integer nonnegative components such that 
$\langle p-q,\gamma\rangle=\chi$ and at least one of the coefficients
$[\mathcal H_\gamma^{[\chi]},y_1^{q_1} y_2^{q_2} 
\mathcal E_\gamma]_1^{(p_1+1,p_2)}$ and
$[\mathcal H_\gamma^{[\chi]},y_1^{q_1} y_2^{q_2}
\mathcal E_\gamma]_2^{(p_1,p_2+1)}$ is nonzero.

Define a subset $\mathfrak Y$ of the set of coefficients of $\mathcal Y,$  
element by element, as follows: 
\begin{itemize}
	\item[i)] $Y_1^{(p_1+1,p_2)},Y_2^{(p_1,p_2+1)}\in\mathfrak Y,$ if 
	$y_1^{p_1+1} y_2^{p_2+1}\in \widetilde{\mathfrak S}_\gamma,$
	$y_1^{p_1} y_2^{p_2}\in \mathfrak S_\gamma;$
	\item[ii)] either $Y_1^{(p_1+1,p_2)}\in\mathfrak Y$ or 
	$Y_2^{(p_1,p_2+1)}\in\mathfrak Y,$ if 
	$y_1^{p_1+1} y_2^{p_2+1}\not\in \widetilde{\mathfrak S}_\gamma,$
	$y_1^{p_1} y_2^{p_2}\in \mathfrak S_\gamma;$	
	\item[iii)] either $Y_1^{(p_1+1,p_2)}\in\mathfrak Y$ or 
	$Y_2^{(p_1,p_2+1)}\in\mathfrak Y,$ if
	$y_1^{p_1+1} y_2^{p_2+1}\in \widetilde{\mathfrak S}_\gamma,$
	$y_1^{p_1} y_2^{p_2}\not\in \mathfrak S_\gamma$ and there exists a vector
	$q=(q_1,q_2)$ with nonnegative integer components such that	
	$\langle p-q,\gamma\rangle=\chi$ and, respectively,
	$[\mathcal H_\gamma^{[\chi]},y_1^{q_1} y_2^{q_2} 
	\mathcal E_\gamma]_2^{(p_1,p_2+1)}\ne0$ or
	$[\mathcal H_\gamma^{[\chi]},y_1^{q_1} y_2^{q_2}
	\mathcal E_\gamma]_1^{(p_1+1,p_2)}\ne0.$
\end{itemize}
\begin{theorem}\ \ 
Given any system~\eqref{Eq1} with a Hamiltonian unperturbed part 
$\mathcal P_\gamma^{[\chi]}=\mathcal H_\gamma^{[\chi]}$ where 
$\mathcal H_\gamma^{[\chi]}=(-\partial_2 H_\gamma^{[\chi+\delta]}, 
\partial_1 H_\gamma^{[\chi+\delta]}),$ and given arbitrary Hamiltonian 
resonant set $\mathfrak S_\gamma,$ Hamiltonian reduced resonant set 
$\widetilde{\mathfrak S}_\gamma,$ and the set $\mathfrak Y$ constructed by the 
above rules, then there exists a near-identity formal 
transformation~\eqref{Eq2} that brings it to the form~\eqref{Eq3}, where all 
coefficients of the perturbation are zero, except for the coefficients from 
$\mathfrak Y$ that have arbitrary values. Moreover, the obtained system is a 
GNF.
\end{theorem}
\begin{proof}\ \
According to Proposition~1, it is enough to show that the set $\mathfrak Y$ is 
a resonant set for the unperturbed part $\mathcal H_\gamma^{[\chi]}.$

Denote $\mathfrak Y_\gamma^{[k+\chi]}=
\{Y_1^{(p_1+1,p_2)},Y_2^{(p_1,p_2+1)}\in\mathfrak Y:\
\langle p,\gamma\rangle=k+\chi\},$ $n_k=|\mathfrak Y_\gamma^{[k+\chi]}|,$
$r_k=|\mathfrak S_\gamma^{[k+\chi]}|,$ and $\widetilde r_k=
|\widetilde{\mathfrak S}_\gamma^{[k+\chi+\delta]}|.$

First, show that for all $k\ge1,$ the number $n_k$ is equal to the number of 
independent resonance equations in g.~d. $k+\chi,$ i.~e., 
$n_k=\dim\mathfrak V_\gamma^{[k+\chi]}-\dim [
\mathcal H_{\gamma}^{[\chi]},\mathfrak V_\gamma^{[k]}],$ where 
$\mathfrak V_\gamma^{[k+\chi]}$ denotes the linear space of VQHPs of g.~d. 
$k+\chi,$ and $[\mathcal H_{\gamma}^{[\chi]},\mathfrak V_\gamma^{[k]}]$ 
denotes the image of the linear operator $[\mathcal H_{\gamma}^{[\chi]},\,.\,]:
\mathfrak V_\gamma^{[k]}\rightarrow\mathfrak V_\gamma^{[k+\chi]}.$

Indeed, by construction, $n_k=r_k+\widetilde r_k.$ In turn, it follows from the
definition of the sets $\mathfrak S_\gamma$ and 
$\widetilde{\mathfrak S}_\gamma$ that $r_k=\dim\mathfrak P_\gamma^{[k+\chi]}-
\dim\mathcal H_\gamma^{[\chi]}(\mathfrak P_\gamma^{[k]})$ and 
$\widetilde r_k=\dim\mathfrak P_\gamma^{[k+\chi+\delta]}-
\dim\bigl(\mathcal H_\gamma^{[\chi]}(\mathfrak P_\gamma^{[k+\delta]})+
\mathrm{Ker\,}\mathcal H_\gamma^{[\chi]}\bigr|_{%
\mathfrak P_\gamma^{[k+\chi+\delta]}}\bigr).$
According to~\eqref{Eq6} and~\eqref{Eq9},
$\dim\mathfrak V_\gamma^{[k+\chi]}=\dim\mathfrak P_\gamma^{[k+\chi+\delta]}+
\dim\mathfrak P_\gamma^{[k+\chi]},$ and
$\dim [\mathcal H_{\gamma}^{[\chi]},\mathfrak V_\gamma^{[k]}]=
\dim\mathcal H_\gamma^{[\chi]}(\mathfrak P_\gamma^{[k]})+
\dim\bigl(\mathcal H_\gamma^{[\chi]}(\mathfrak P_\gamma^{[k+\delta]})+
\mathrm{Ker\,}\mathcal H_\gamma^{[\chi]}\bigr|_{%
\mathfrak P_\gamma^{[k+\chi+\delta]}}\bigr).$
Hence, $n_k=r_k+\widetilde r_k=\dim\mathfrak V_\gamma^{[k+\chi]}-
\dim [\mathcal H_{\gamma}^{[\chi]},\mathfrak V_\gamma^{[k]}].$

By construction, all the $n_k$ elements of the set 
$\mathfrak Y_\gamma^{[k+\chi]}$ can be uniquely expressed from 
the system of $n_k$ linear equations~\eqref{Eq13} in terms of coefficients of 
the QHPs $\widetilde F_\gamma^{[k+\chi+\delta]}, 
\widetilde G_\gamma^{[k+\chi]}$ from~\eqref{Eq10} (all but the $n_k$ 
coefficients of the perturbation of the GPHNF are zero in g.~d. $k+\chi$). 
Therefore, it remains to show transformations that eliminate the 
coefficients that do not belong to $\mathfrak Y,$ in cases~ii) and~iii).

ii) Let $Y_1^{(p_1+1,p_2)}\in\mathfrak Y,$
$y_1^{p_1+1} y_2^{p_2+1}\not\in \widetilde{\mathfrak S}_\gamma,$
$y_1^{p_1} y_2^{p_2}\in \mathfrak S_\gamma.$ Thereby, 
$Y_2^{(p_1,p_2+1)}\not\in\mathfrak Y$ and $(Y_1^{(p_1+1,p_2)} 
y_1^{p_1+1} y_2^{p_2}, Y_2^{(p_1,p_2+1)} y_1^{p_1} y_2^{p_2+1})= 
\widetilde{G}^{(p_1,p_2)} y_1^{p_1} y_2^{p_2} \mathcal E_\gamma.$
It follows then from the expansion
$$y_1^{p_1} y_2^{p_2} \mathcal E_\gamma= \frac{\gamma_2}{p_1+1} 
\Bigl(-(p_2+1)y_1^{p_1+1}y_2^{p_2},(p_1+1)y_1^{p_1}y_2^{p_2+1}\Bigr)+ 
\frac{\gamma_1(p_1+1)+\gamma_2(p_2+1)}{p_1+1}\Bigl(y_1^{p_1+1}y_2^{p_2},0
\Bigr)$$
that the coefficient $Y_2^{(p_1,p_2+1)}$ can be set to zero by adding the term 
$-\gamma_2 \widetilde{G}^{(p_1,p_2)}(p_1+1)^{-1} y_1^{p_1+1} y_2^{p_2+1}$ 
to~$\widetilde{F}_\gamma^{[k+\chi+\delta]}$ in the proof of Theorem~1, which 
is possible, since $y_1^{p_1+1} y_2^{p_2+1}\not\in
\widetilde{\mathfrak S}_\gamma.$

Similarly, in case where $Y_2^{(p_1,p_2+1)}\in\mathfrak Y,$
$y_1^{p_1+1} y_2^{p_2+1}\not\in \widetilde{\mathfrak S}_\gamma,$
$y_1^{p_1} y_2^{p_2}\in \mathfrak S_\gamma,$ we eliminate the coefficient
$Y_1^{(p_1+1,p_2)}\not\in\mathfrak Y.$

iii) Let $Y_1^{(p_1+1,p_2)}\in\mathfrak Y,$
$y_1^{p_1+1} y_2^{p_2+1}\in \widetilde{\mathfrak S}_\gamma,$
$y_1^{p_1} y_2^{p_2}\not\in \mathfrak S_\gamma,$
and let there exist a vector $q=(q_1,q_2)$ with nonnegative integer components 
$q_1,q_2$ such that $\langle p-q,\gamma\rangle=\chi,$ 
$[\mathcal H_\gamma^{[\chi]},y_1^{q_1} y_2^{q_2}
\mathcal E_\gamma]_2^{(p_1,p_2+1)}\ne0.$ Thereby, 
$Y_2^{(p_1,p_2+1)}\not\in\mathfrak Y$ and
$$\Bigl(Y_1^{(p_1+1,p_2)} y_1^{p_1+1} y_2^{p_2},
Y_2^{(p_1,p_2+1)} y_1^{p_1} y_2^{p_2+1}\Bigr)=
\widetilde{F}^{(p_1+1,p_2+1)} \Bigl(-(p_2+1) y_1^{p_1+1}y_2^{p_2},
(p_1+1) y_1^{p_1}y_2^{p_2+1}\Bigr).$$
Then the coefficient $Y_2^{(p_1,p_2+1)}$ can be set to zero by using a
transformation of the form~\eqref{Eq2} where 
$\mathcal Q=C y_1^{q_1} y_2^{q_2} \mathcal E$ with an appropriate coefficient 
$C.$ Herewith, all the extra terms can be taken into account by adding in a 
suitable way the terms that do not belong to $\widetilde{\mathfrak S}_\gamma$ 
and $\mathfrak S_\gamma$ to the QHPs $\widetilde{F}_\gamma^{[k+\chi+\delta]}$ 
and $\widetilde{G}_\gamma^{[k+\chi]}$ respectively (see the proof of 
Theorem~1). 

Similarly, if $Y_2^{(p_1,p_2+1)}\in\mathfrak Y,$
$y_1^{p_1+1} y_2^{p_2+1}\in \widetilde{\mathfrak S}_\gamma,$
$y_1^{p_1} y_2^{p_2}\not\in \mathfrak S_\gamma,$ and 
$[\mathcal H_\gamma^{[\chi]},y_1^{q_1} y_2^{q_2}
\mathcal E_\gamma]_1^{(p_1+1,p_2)}\ne0,$ we eliminate the coefficient
$Y_1^{(p_1+1,p_2)}\not\in\mathfrak Y.$
\end{proof}
\begin{remark}\ \
If there are several pairs of coefficients $Y_1^{(p_1+1,p_2)}, 
Y_2^{(p_1,p_2+1)}$ in one and the same g.~d. that fit case iii), then there 
exists the same number of monomials $y_1^{q_1} y_2^{q_2}$ that satisfy the 
conditions described for this case, and the coefficients for the 
transformations are obtained from an algebraic system with nonzero 
determinant that expresses the equality to zero of the corresponding 
coefficients of $\mathcal Y.$
\end{remark}
\section{GNFs of systems the Hamiltonian unperturbed part of which has 
monomial components}
In this section, using Theorems~1 and~2, we compute GNFs for systems with a 
Hamiltonian unperturbed part represented by a vector monomial, i.~e. a vector 
with the monomial components. The results are compared with the known GNFs, 
obtained earlier by Takens~\cite{takens}, Baider and Sanders~\cite{sanders2}, 
Basov et~al.~\cite{basov03,basov,fed,skit}.
\subsection{The unperturbed part $(x_2^{m-1},0)$ with $m\ge2$}
Consider system~\eqref{Eq8} where $\mathcal H^{[\chi]}_\gamma=(x_2^{m-1},0),$ 
$H^{[\chi+\delta]}_\gamma=-\cfrac{x_2^{m}}{m},$ $\gamma=(1,1),$ $\chi=m-2,$ 
$m\ge2.$

In this case, the space of resonant polynomials has the view (see~\cite{BV})
$$\mathfrak R_\gamma=Lin(\{x_1^{p_1} x_2^{p_2}:\ p_2=\overline{0,m-2}\,\}).$$
$\mathfrak R_\gamma$ does not contain integrals of
$\mathcal H_\gamma^{[\chi]}$ of degree higher than $m,$ that is, every 
Hamiltonian reduced resonant set is a Hamiltonian resonant set, in degrees 
higher than $m.$
\begin{corollary}\ \ 
Given any system~\eqref{Eq1} with the unperturbed part $(x_2^{m-1},0)$ where 
$m\ge2,$ then there exists a near-identity formal transformation~\eqref{Eq2} 
that brings it into the GNF 
\begin{equation}
\label{Eq14}
\begin{split}
&\dot y_1=y_2^{m-1}+
\sum_{i=m}^\infty\sum_{j=0}^{m-3}Y_1^{(i-j,j)}y_1^{i-j}y_2^j
+y_1^{2} y_2^{m-2} \sum_{i=0}^\infty
Y_1^{(i+2,m-2)} y_1^{i},\\
&\dot y_2=
\sum_{i=m}^\infty\sum_{j=0}^{m-2}Y_2^{(i-j,j)}y_1^{i-j}y_2^j+
y_1 y_2^{m-1} \sum_{i=0}^\infty
Y_2^{(i+1,m-1)} y_1^{i}
\end{split}
\end{equation}
where for each $i\ge0,$ we take either $Y_1^{(i+2,m-2)}=0$ or 
$Y_2^{(i+1,m-1)}=0.$
\end{corollary}
\begin{proof}\ \ 
For all $k>m,$ choose the Hamiltonian resonant sets
$$\mathfrak S_\gamma^{[k]},\widetilde{\mathfrak S}_\gamma^{[k]}=
\{x_1^{p_1} x_2^{p_2}:\ p_2=\overline{0,m-2};\ |p|=k\,\}.$$
Then GPHNF~\eqref{Eq10} takes the form~\eqref{Eq14} with 
$Y_1^{(i+2,m-2)}=Y_2^{(i+1,m-1)}=\widetilde G^{(i+1,m-2)}.$ Hence, the 
corollary follows from Theorem~2.
\end{proof}

For $m=2,$ formula~\eqref{Eq14} gives \textit{the Takens normal 
form}~\cite{takens}, and for $m=3,$ it gives~\cite[Th.~11]{fed}.
\subsection{The unperturbed part $(-m x_1^{l} x_2^{m-1}, l x_1^{l-1} x_2^{m})$ 
with $l>m\ge1$}
Consider system~\eqref{Eq8} where 
$\mathcal H^{[\chi]}_\gamma=(-m x_1^{l} x_2^{m-1}, l x_1^{l-1} x_2^{m}),$
$H^{[\chi+\delta]}_\gamma={x_1^l x_2^m},$ $\gamma=(1,1),$ $\chi=l+m-2,$ 
$l>m\ge1,$ and $\text{GCD}(l,m)=d.$

In this case, the space of resonant polynomials has the view (see~\cite{BV})
$$\mathfrak R_\gamma=Lin\bigl(\{x_1^{p_1} x_2^{p_2}:\ p_1=\overline{0,l-2},
\text{ or } p_2=\overline{0,m-2},\text{ or } p_1=rl/d-1,\ p_2=rm/d-1,\ 
r\ge d\,\}\bigr).$$
$\mathfrak R_\gamma$ does not contain integrals of 
$\mathcal H_\gamma^{[\chi]}$ of degree higher than $l+m,$ that is, every 
Hamiltonian reduced resonant set is a Hamiltonian resonant set, in degrees 
higher than $l+m.$
\begin{corollary}\ \ 
Given any system~\eqref{Eq1} with the unperturbed part $(-m x_1^{l} x_2^{m-1}, 
l x_1^{l-1} x_2^{m})$ where $l>m\ge1$ and $\mathrm{GCD}(l,m)=d,$ then there 
exists a near-identity formal transformation~\eqref{Eq2} that brings it into 
the GNF
\begin{equation}\label{Eq15}
\begin{split}
&\dot y_1=-m y_1^{l} y_2^{m-1}+
 \sum_{k=l+m}^\infty\Bigl(\sum_{i=0}^{l-2} Y_1^{(i,k-i)}y_1^i y_2^{k-i}+
 \sum_{j=0}^{m-3} Y_1^{(k-j,j)}y_1^{k-j} y_2^j\Bigr)+\\
&\qquad +y_1^{l+2} y_2^{m-2}\sum_{i=0}^\infty Y_1^{(i+l+2,m-2)}y_1^{i}+
 y_1^{l-1} y_2^{m+1} \sum_{j=0}^\infty Y_1^{(l-1,m+j+1)}y_2^{j}+\\
&\qquad\qquad +\sum_{r=d+1}^\infty Y_1^{(rl/d,rm/d-1)} y_1^{rl/d} y_2^{rm/d-1}+
 \sum_{s=d+1+[\frac{3d-1}{l+m}]}^\infty Y_1^{(sl/d-1,sm/d-2)} y_1^{sl/d-1} y_2^{sm/d-2},\\
&\dot y_2=l y_1^{l-1} y_2^{m}+
 \sum_{k=l+m}^\infty\Bigl(\sum_{i=0}^{l-3} Y_2^{(i,k-i)}y_1^i y_2^{k-i}+
 \sum_{j=0}^{m-2} Y_2^{(k-j,j)}y_1^{k-j} y_2^j\Bigr)+\\
&\qquad +y_1^{l+1} y_2^{m-1} \sum_{i=0}^\infty Y_2^{(i+l+1,m-1)}y_1^{i}+
 y_1^{l-2} y_2^{m+2}\sum_{j=0}^\infty Y_2^{(l-2,m+j+2)}y_2^{j}+\\
&\qquad\qquad +\sum_{r=d+1}^\infty Y_2^{(rl/d-1,rm/d)} y_1^{rl/d-1} y_2^{rm/d}+
 \sum_{s=d+1+[\frac{3d-1}{l+m}]}^\infty Y_2^{(sl/d-2,sm/d-1)} y_1^{sl/d-2} y_2^{sm/d-1}
\end{split}
\end{equation}
where for each $i,j\ge0,$ $r\ge d+1$ and $s\ge d+1+[\frac{3d-1}{l+m}],$ we take
either $Y_1^{(i+l+2,m-2)}=0$ or $Y_2^{(i+l+1,m-1)}=0,$ either 
$Y_1^{(l-1,m+j+1)}=0$ or $Y_2^{(l-2,m+j+2)}=0,$ either $Y_1^{(rl/d,rm/d-1)}=0$ 
or $Y_2^{(rl/d-1,rm/d)}=0,$ and in case $m=1,$ we take 
$Y_2^{(sl/d-2,sm/d-1)}=0,$ and in case $m\ge2$ we take either 
$Y_1^{(sl/d-1,sm/d-2)}=0$ or $Y_2^{(sl/d-2,sm/d-1)}=0.$
\end{corollary}
\begin{proof} \ \
For all $k>l+m,$ choose the Hamiltonian resonant sets
\begin{multline}\notag
\mathfrak S_\gamma^{[k]},\widetilde{\mathfrak S}_\gamma^{[k]}=
\{x_1^{p_1} x_2^{p_2}:\ p_1=\overline{0,l-2}, \text{ or } 
p_2=\overline{0,m-2}, \text{ or }\\ p_1=rl/d-1,\ p_2=rm/d-1,\ r\ge d;\  
|p|=k\,\}.
\end{multline}
Then GPHNF~\eqref{Eq10} takes the form~\eqref{Eq15} where
$Y_1^{(l-1,m+j+1)}=Y_2^{(l-2,m+j+2)}=\widetilde G^{(l-2,m+j+1)},$
$Y_1^{(i+l+2,m-2)}=Y_2^{(i+l+1,m-1)}=\widetilde G^{(i+l+1,m-2)},$
$Y_1^{(rl/d,rm/d-1)}=Y_2^{(rl/d-1,rm/d)}=\widetilde G^{(rl/d-1,rm/d-1)},$ and 
$Y_1^{(sl/d-1,sm/d-2)}=(1-sm/d)\widetilde F^{(sl/d-1,sm/d-1)},$
$Y_2^{(sl/d-2,sm/d-1)}=(sl/d-1)\widetilde F^{(sl/d-1,sm/d-1)}.$

Transformation~\eqref{Eq2}, with 
$\mathcal Q= C y_1^{s l/d-l-1} y_2^{s m/d-m-1} \mathcal E$ where 
$C=\mathrm{const},$ changes the coefficients 
$Y_1^{(sl/d-1,sm/d-2)}$ and $Y_2^{(sl/d-2,sm/d-1)}$ by $C(1-m)(l+m)$ and 
$C(l-1)(l+m)$ respectevely. Hence, the corollary follows from Theorem~2.
\end{proof}

For $l=2,\,m=1,$ formula~\eqref{Eq15} gives~\cite[Th.~7]{skit}
for $\alpha=-1/2.$  
\subsection{The unperturbed part $(-x_1^{m} x_2^{m-1}, x_1^{m-1} x_2^{m})$ 
with $m\ge1$}
Consider system~\eqref{Eq8} where $\mathcal H^{[\chi]}_\gamma=
(-x_1^{m} x_2^{m-1}, x_1^{m-1} x_2^{m}),$ $H^{[\chi+\delta]}_\gamma=
{x_1^m x_2^m}/{m},$ $\gamma=(1,1),$ $\chi=2m-2,$ and $m\ge1.$

In this case, the space of resonant polynomials has the view (see~\cite{BV})
$$\mathfrak R_\gamma=Lin\bigl(\{x_1^{p_1} x_2^{p_2}:\ p_1=\overline{0,m-2}, 
\text{ or } p_2=\overline{0,m-2}, \text{ or } p_1=p_2\,\}\bigr).$$

Since monomials $x_1^k x_2^k$ $(k\ge0)$ are integrals of the unperturbed part, 
such monomials are absent in the Hamiltonian reduced resonant set, in degrees 
higher than $2m.$
\begin{corollary}\ \ Given any system~\eqref{Eq1} with the unperturbed part
$(-x_1^{m} x_2^{m-1}, x_1^{m-1} x_2^{m})$ where $m\ge1,$ then there exists a 
near-identity formal transformation~\eqref{Eq2} that brings it into the GNF 
\begin{equation}
\label{Eq16}
\begin{split}
&\dot y_1=-y_1^{m} y_2^{m-1}+\sum_{k=2m}^\infty\Bigl(\sum_{i=0}^{m-2} Y_1^{(i,k-i)}y_1^i y_2^{k-i}+
\sum_{j=0}^{m-3} Y_1^{(k-j,j)}y_1^{k-j} y_2^j\Bigr)+\\
&\qquad+
y_1^{m+2} y_2^{m-2}\sum_{i=0}^\infty Y_1^{(i+m+2,m-2)} y_1^{i}
+y_1^{m-1} y_2^{m+1} \sum_{j=0}^\infty Y_1^{(m-1,m+j+1)} y_2^{j}+\\
&\qquad\qquad+y_1^{m} y_2^{m}\sum_{r=0}^\infty Y_1^{(r+m+1,r+m)} y_1^{r+1} y_2^{r},\\
&\dot y_2=y_1^{m-1} y_2^{m}+\sum_{k=2m}^\infty\Bigl(\sum_{i=0}^{m-3} Y_2^{(i,k-i)}y_1^i y_2^{k-i}+
\sum_{j=0}^{m-2} Y_2^{(k-j,j)}y_1^{k-j} y_2^j\Bigr)+\\
&\qquad+y_1^{m+1} y_2^{m-1} \sum_{i=0}^\infty Y_2^{(i+m+1,m-1)} y_1^{i}+
y_1^{m-2} y_2^{m+2}\sum_{j=0}^\infty Y_2^{(m-2,m+j+2)} y_2^{j}+\\
&\qquad\qquad+y_1^{m} y_2^{m}\sum_{r=0}^\infty Y_2^{(r+m,r+m+1)} y_1^{r} y_2^{r+1}
\end{split}
\end{equation}
where for each $i,j,r\ge0,$ we take either $Y_1^{(i+m+2,m-2)}=0$ or 
$Y_2^{(i+m+1,m-1)}=0,$ either $Y_1^{(m-1,m+j+1)}=0$ or $Y_2^{(m-2,m+j+2)}=0,$
and either $Y_1^{(r+m+1,r+m)}=0$ or $Y_2^{(r+m,r+m+1)}=0.$
\end{corollary}
\begin{proof}\ \ For all $k>2m,$ choose
\begin{equation}
\notag
\begin{split}
&\mathfrak S_\gamma^{[k]}=
\{x_1^{p_1} x_2^{p_2}:\ p_1=\overline{0,m-2}, \text{ or } 
p_2=\overline{0,m-2}, \text{ or } p_1=p_2;\ |p|=k\,\},\\
&\widetilde{\mathfrak S}_\gamma^{[k]}=
\{x_1^{p_1} x_2^{p_2}:\ p_1=\overline{0,m-2} \text{ or } 
p_2=\overline{0,m-2};\ |p|=k\,\}.
\end{split}
\end{equation}
Then GPHNF~\eqref{Eq10} takes the form~\eqref{Eq16} where
$Y_1^{(i+m+2,m-2)}=Y_2^{(i+m+1,m-1)}=\widetilde G^{(i+m+1,m-2)},$
$Y_1^{(m-1,m+j+1)}=Y_2^{(m-2,m+j+2)}=\widetilde G^{(m-2,m+j+1)},$
$Y_1^{(r+m+1,r+m)}=Y_2^{(r+m,r+m+1)}=\widetilde G^{(r+m,r+m)}.$
Hence, the corollary follows from Theorem~2.
\end{proof}
\subsection{The unperturbed part $(\pm x_2^{m-1}, x_1^{l-1})$ with 
$l\ge m\ge2$}
Consider system~\eqref{Eq8} where $\mathcal H^{[\chi]}_\gamma=
(\pm x_2^{m-1}, x_1^{l-1}),$ $H^{[\chi+\delta]}_\gamma=x_1^l/l\mp x_2^m/m,$ 
$\gamma=(m/d,l/d),$ $\chi=(lm-l-m)/d,$ $l\ge m\ge2,$ and $d=\mathrm{GCD}(l,m).$
\begin{lemma}
\label{Lmbin}\ \
Minimal Hamiltonian resonant and reduced resonant sets in g.~d. $k>lm/d$ have 
the view
\begin{equation}
\notag
\begin{split}
&\mathfrak S_\gamma^{[k]}=\{x_1^{p_1-r_{p} l}x_2^{p_2+r_{p} m}:\ 
p_1\not\equiv-1\ \mathrm{mod}\,l,\ p_2=\overline{0,m-2}\,\},\\
&\widetilde{\mathfrak S}_\gamma^{[k]}=
\{x_1^{p_1-\widetilde r_{p} l}x_2^{p_2+\widetilde r_{p} m}:\
p_1\not\equiv-1,0\ \mathrm{mod}\,l,\ p_2=0 \text{ or }
p_1\not\equiv-1\ \mathrm{mod}\,l,\ p_2=\overline{1,m-2}\,\}
\end{split}
\end{equation}
where $\langle p,\gamma\rangle=k,$ and $r_{p},\widetilde r_{p}$ are arbitrary 
integers such that $0\le r_{p},\widetilde r_{p}\le [p_1/l].$
\end{lemma}
\begin{proof}\ \
Let $R=\sum\limits^\infty_{p_1,p_2=0} R^{(p_1,p_2)} x_1^{p_1} 
x_2^{p_2}\in\mathfrak R_\gamma.$

According to formula~\eqref{Eq7} $\left.\mathcal H_\gamma^{[\chi]}\right.^*=
x_2({\partial^{l-1}}/{\partial x_1^{l-1}})\pm
x_1({\partial^{m-1}}/{\partial x_2^{m-1}}),$ thus
\begin{equation}
\notag
\sum^\infty_{i=l-1}\sum^\infty_{j=0}
(l-1)! C^{l-1}_i R^{(i,j)} x_1^{i-l+1} x_2^{j+1}\pm
\sum^\infty_{i=0}\sum^\infty_{j=m-1}
(m-1)! C^{m-1}_j R^{(i,j)} x_1^{i+1} x_2^{j-m+1} =0,
\end{equation}
thence, setting the coefficients of $x_1^{i+1},x_2^{j+1},$ and 
$x_1^{i+1}x_2^{j+1}$ to zero, we obtain the equations
\begin{equation}
\label{Eq17}
R^{(i,m-1)}=0,\quad R^{(l-1,j)}=0,\quad(l-1)! C^{l-1}_{i+l} R^{(i+l,j)}\pm
(m-1)! C^{m-1}_{j+m} R^{(i,j+m)}=0\qquad(i,j\ge0).
\end{equation}
Hence, by induction, we find that $R^{(i,km-1)},R^{(kl-1,j)}=0$ for all 
$k\ge1.$

It also follows from equations~\eqref{Eq17} that any resonant polynomial $R$ 
is uniquely defined by its coefficients $R^{(p_1-r_{p}l,p_2+r_{p}m)}$ of the 
monomials $x_1^{p_1-r_{p}l}x_2^{p_2+r_{p}m}$ where $p_2\le m-2$ 
$(0\le r_{p}\le [p_1/l]).$ Thus, set $\mathfrak S_\gamma^{[k]}$ of such 
monomials is a minimal Hamiltonian resonant set.

For the Hamiltonian reduced resonant set, the lemma follows from the fact that 
any quasi-homogeneous polynomial integral for $\mathcal H^{[\chi]}_\gamma$ is 
a power of $H^{[\chi+\delta]}_\gamma.$
\end{proof}

For each given $k\in \mathbb Z,$ denote $\theta[k]=0$ if $k<0,$ and 
$\theta[k]=1$ if $k\ge0.$
\begin{corollary}\ \ 
Given any system~\eqref{Eq1} with the unperturbed part $(\pm x_2^{m-1}, 
x_1^{l-1})$ where $l\ge m\ge2,$ then there exists a near-identity formal 
transformation~\eqref{Eq2} that brings it into the GNF
\begin{equation}
\label{Eq18}
\begin{split}
&\dot y_1=\pm y_2^{m-1}+\sum_{j=1}^{m-2}y_2^{j-1}\biggl(
\sum_{\substack{i\not\equiv0,-1\,\mathrm{mod}\, l\\ i>l(1-j/m)}}
Y_1^{(i,j-1)} y_1^{i} +
\sum_{r=1+\theta[m-jl]}^\infty Y_1^{(r l-1,j-1)} y_1^{rl-1}+\\
&\qquad+\sum_{s=1}^\infty Y_1^{(sl,j-1)} y_1^{sl}\biggr)+
\sum_{\substack{i\not\equiv0\,\mathrm{mod}\, l\\  i>l/m}} Y_1^{(i,m-2)} 
y_1^{i}y_2^{m-2},\\
&\dot y_2=y_1^{l-1}+\sum_{j=1}^{m-2}y_2^{j}\biggl(
\sum_{\substack{i\not\equiv0,-1\,\mathrm{mod}\, l\\ i>l(1-j/m)}}
Y_2^{(i-1,j)}y_1^{i-1}+\sum_{r=1+\theta[m-jl]}^\infty Y_2^{(rl-2,j)} 
y_1^{rl-2}+\\
&\qquad+\sum_{s=1}^\infty Y_2^{(sl-1,j)} y_1^{sl-1}\biggr)+
\sum_{\substack{i\not\equiv0,-1\,\mathrm{mod}\, l\\ i>l}} Y_2^{(i-1,0)} 
y_1^{i-1}+\sum_{\substack{i\not\equiv0\,\mathrm{mod}\, l\\  i>l/m}}
Y_2^{(i-1,m-1)} y_1^{i-1}y_2^{m-1}
\end{split}
\end{equation}
where for each $i\not\equiv0\mod l,$ $i>l/m,$ $j=\overline{1,m-2},$ 
$r\ge 1+\theta[m-jl],$ and $s\ge1,$ we take either $Y_1^{(i,m-2)}=0$ or 
$Y_2^{(i-1,m-1)}=0,$ either $Y_1^{(r l-1,j-1)}=0$ or $Y_2^{(rl-2,j)}=0,$
either $Y_1^{(sl,j-1)}=0$ or $Y_2^{(sl-1,j)}=0,$ except for the case where 
$l=m,$ in which for $j=m-2,$ we take $Y_1^{(sm,m-3)}=0$ in the pairs 
$\{Y_1^{(sm,m-3)},Y_2^{(sm-1,m-2)}\}.$ 
\end{corollary}
\begin{proof}\ \ 
For all $k>lm/d,$ choose the Hamiltonian resonant and reduced resonant sets 
given by Lemma~\ref{Lmbin} with all $r_{p},\widetilde r_{p}=0:$
\begin{equation}
\notag
\begin{split}
&\mathfrak S_\gamma^{[k]}=\{x_1^{p_1} x_2^{p_2}:\ p_1\not\equiv-1\ 
\mathrm{mod}\,l,\ p_2=\overline{0,m-2};\ \langle p,\gamma\rangle=k\,\},\\
&\widetilde{\mathfrak S}_\gamma^{[k]}=\{x_1^{p_1} x_2^{p_2}:\ 
p_1\not\equiv-1,0\ \mathrm{mod}\,l,\ p_2=0 \text{ or }
p_1\not\equiv-1\ \mathrm{mod}\,l,\ p_2=\overline{1,m-2};\ 
\langle p,\gamma\rangle=k\,\}.
\end{split}
\end{equation}
Then GPHNF~\eqref{Eq10} takes the form~\eqref{Eq18} where
$Y_1^{(i,m-2)}=Y_2^{(i-1,m-1)}=\widetilde G^{(i-1,m-2)},$
$Y_1^{(r l-1,j-1)}=Y_2^{(rl-2,j)}=\widetilde G^{(rl-2,j-1)},$
$Y_1^{(sl,j-1)}=-j\widetilde F^{(sl,j)},$ and $Y_2^{(sl-1,j)}= s l 
\widetilde F^{(sl,j)}.$

Transformation~\eqref{Eq2} with $\mathcal Q=C y_1^{sl-l} y_2^{j} 
\mathcal E$ where $C=\mathrm{const}$ changes the coefficients 
$Y_1^{(sl,j-1)}$ and $Y_2^{(sl-1,j)}$ by $-C j$ and $C(l-j-2)$ respectively. 
Hence, for $l>m,$ as well as for $l=m$ and $j=\overline{1,m-3},$ we can set 
either coefficient in the pair $\{Y_1^{(sl,j-1)}, Y_2^{(sl-1,j)}\}$ to zero by 
choosing $C.$ And only in case where $l=m$ and $j=m-2,$ the coefficient 
$Y_2^{(sm-1,m-2)}$ does not change under this transformation. In this case, we 
zero the coefficent $Y_1^{(sm,m-3)}$ in system~\eqref{Eq18} by choosing 
$C=Y_1^{(sm,m-3)}/(m-3).$ Hence, the corollary follows from Theorem~2.
\end{proof}

For $m=2,$ GNF~\eqref{Eq18} is the so called \textit{second order 
Takens-Bogdanov normal form} that was obtained by Baider and Sanders 
in~\cite{sanders2}. In particular, for $l=3,\,m=2,$ formula~\eqref{Eq18} 
agrees with~\cite[Th.~4]{basov}, and in case $l=4,\,m=2,$ it is consistent 
with~\cite[Th.~3]{basov03}.
\newpage
\end{document}